\documentclass[a4paper,12pt]{article}
\usepackage{amsmath}
\usepackage{amsthm}
\usepackage{amssymb}
\usepackage{enumerate}
\usepackage{url}

\title{A pathwise interpretation of the Gorin-Shkolnikov identity}

\author{Yuu Hariya\thanks{Mathematical Institute, 
Tohoku University, Aoba-ku, Sendai 980-8578, Japan. }}

\date{\empty}
\numberwithin{equation}{section}

\theoremstyle{plain}

\newtheorem{thm}{Theorem}[section]

\newtheorem{prop}{Proposition}[section]

\theoremstyle{definition}

\theoremstyle{remark}
\newtheorem{rem}{Remark}[section]

\begin{document}

\def\N {\mathbb{N}}
\def\R {\mathbb{R}}
\def\Q {\mathbb{Q}}

\def\calF {\mathcal{F}}

\def\kp {\kappa}

\def\ind {\boldsymbol{1}}

\def\al {\alpha }
\def\la {\lambda }
\def\ve {\varepsilon}
\def\Om {\Omega}

\def\v {v}

\def\ga {\gamma }

\def\W {\mathbb{W}}
\def\H {\mathbb{H}}
\def\A {\mathcal{A}}

\newcommand\ND{\newcommand}
\newcommand\RD{\renewcommand}

\ND\lref[1]{Lemma~\ref{#1}}
\ND\tref[1]{Theorem~\ref{#1}}
\ND\pref[1]{Proposition~\ref{#1}}
\ND\sref[1]{Section~\ref{#1}}
\ND\ssref[1]{Subsection~\ref{#1}}
\ND\aref[1]{Appendix~\ref{#1}}
\ND\rref[1]{Remark~\ref{#1}} 
\ND\cref[1]{Corollary~\ref{#1}}
\ND\eref[1]{Example~\ref{#1}}
\ND\fref[1]{Fig.\ {#1} }
\ND\lsref[1]{Lemmas~\ref{#1}}
\ND\tsref[1]{Theorems~\ref{#1}}
\ND\dref[1]{Definition~\ref{#1}}
\ND\psref[1]{Propositions~\ref{#1}}
\ND\rsref[1]{Remarks~\ref{#1}}
\ND\sssref[1]{Subsections~\ref{#1}}

\ND\pr{\mathbb{P}}
\ND\ex{\mathbb{E}}
\ND\br{W}
\ND\bes{r}
\ND\loc{l}

\ND\be[1]{R^{(#1)}}

\ND\E[1]{\mathcal{E}^{#1}}
\ND\no[2]{\|{#1}\|_{#2}}
\ND\nt[2]{\|{#1}\|_{#2}}
\ND\tr[2]{T^{#1}_{#2}}
\ND\si{\mathcal{S}}
\ND\lhs[1]{\log\ex\left[e^{{#1}(\br)}\right]}
\ND\inner[2]{({#1},{#2})_{\H }}
\ND\cpl[2]{\langle {#1},{#2}\rangle}

\def\thefootnote{{}}

\maketitle 

\begin{abstract}
In a recent paper by Gorin and Shkolnikov (2016), they have found, as 
a corollary to their result relevant to random matrix theory, that 
the area below a normalized Brownian excursion minus one half of 
the integral of the square of its total local time, is identical 
in law with a centered Gaussian random variable with variance 
$1/12$. 
In this note, we give a pathwise interpretation to 
their identity; Jeulin's identity connecting normalized Brownian excursion 
and its local time plays an essential role in the exposition. 
\footnote{E-mail: hariya@math.tohoku.ac.jp}
\end{abstract}

\section{Introduction}\label{;intro}

Let $\bes =\{ \bes _{t}\} _{0\le t\le 1}$ be a normalized Brownian 
excursion, that is, it is identical in law with a standard $3$-dimensional 
Bessel bridge, which has the duration $[0,1]$, and starts from and ends at 
the origin; see e.g., \cite[Section~\thetag{2.2}]{by} and references therein 
for the definition of normalized Brownian excursion and its equivalence in 
law with standard $3$-dimensional Bessel bridge. We denote by 
$\loc =\{ \loc _{x}\} _{x\ge 0}$ the total local time process of $\bes $; 
namely, by the occupation time formula, two processes 
$\bes $ and $\loc $ are related via 
\begin{align}\label{;defH}
 H(x):=\int _{0}^{1}\ind _{\{ \bes _{t}\le x\} }\,dt
 =\int _{0}^{x}\loc _{y}\,dy \quad \text{for all $x\ge 0$, a.s.}
\end{align}
In a recent paper \cite{gs}, Gorin and Shkolnikov have found the following 
remarkable identity in law as a corollary to one of their results: 
\begin{thm}[\cite{gs}, Corollary~2.15]\label{;mt}
 The random variable $X$ defined by 
 \begin{align*}
  X:=\int _{0}^{1}\bes _{t}\,dt-\frac{1}{2}\int _{0}^\infty 
 \left( \loc _{x}\right) ^{2}dx
 \end{align*}
 is a centered Gaussian random variable with variance $1/12$. 
\end{thm}
In \cite{gs}, they have shown that the expected value of the 
trace of a random operator indexed by $T>0$, arising from 
random matrix theory, admits the representation 
\begin{align*}
 \sqrt{\frac{2}{\pi T^{3}}}
 \ex \left[ 
 \exp \left( -\frac{T^{3/2}}{2}X\right) 
 \right]
\end{align*}
for any $T>0$; in comparison of this expression with the existing literature 
asserting that the expected value is equal to 
$\sqrt{2/(\pi T^{3})}\exp \left( {T^{3}/96}\right)$ for every $T>0$, 
they have obtained 
\tref{;mt} by the analytic continuation and 
the uniqueness of characteristic functions. 

In this note, we give a proof of \tref{;mt} without relying on random 
matrix theory; Jeulin's identity in law 
(\cite[p.\,264]{jeu}, \cite[Proposition~3.6]{by}): 
\begin{align}\label{;ji}
 \{ \bes _{t}\} _{0\le t\le 1}
 \stackrel{(d)}{=}\left\{ 
 \frac{1}{2}\loc _{H^{-1}(t)}
 \right\} _{0\le t\le 1}
\end{align}
with 
\begin{align*}
 H^{-1}(t):=\inf \left\{ 
 x\ge 0;\,H(x)\ge t
 \right\} , 
\end{align*}
plays a central role in the proof. 

\section{Proof of \tref{;mt}}\label{;prf}

In this section, we give a proof of \tref{;mt} and provide 
some relevant results. 

\begin{proof}[Proof of \tref{;mt}]
 Recall from the representation of $\bes $ by means of a stochastic 
 differential equation (see, e.g., 
 \cite[Chapter~XI, Exercise~\thetag{3.11}]{ry}) that the process 
 $\br =\{ \br _{t}\} _{0\le t\le 1}$ defined by 
 \begin{align}\label{;brm}
  \br _{t}:=\bes _{t}-\int _{0}^{t}\frac{ds}{\bes _{s}}
  +\int _{0}^{t}\frac{\bes _{s}}{1-s}\,ds
 \end{align}
 is a standard Brownian motion. We integrate both sides over $[0,1]$ 
 and use Fubini's theorem on the right-hand side to see that 
 \begin{align*}
  \int _{0}^{1}\br _{t}\,dt
  =\int _{0}^{1}\bes _{t}\,dt
  -\int _{0}^{1}\frac{ds}{\bes _{s}}\int _{s}^{1}dt
  +\int _{0}^{1}ds\,\frac{\bes _{s}}{1-s}\int _{s}^{1}dt, 
 \end{align*}
 which entails 
 \begin{align}\label{;gauss1}
  \frac{1}{2}\int _{0}^{1}\br _{t}\,dt
  =\int _{0}^{1}\bes _{t}\,dt
  -\frac{1}{2}\int _{0}^{1}\frac{1-t}{\bes _{t}}\,dt. 
 \end{align}
 Note that the left-hand side is a centered Gaussian random variable with 
 variance $1/12$. By Jeulin's identity \eqref{;ji}, the right-hand side 
 of \eqref{;gauss1} is identical in law with 
 \begin{align}
  \frac{1}{2}\int _{0}^{1}\loc _{H^{-1}(t)}\,dt
  -\int _{0}^{1}\frac{1-t}{\loc _{H^{-1}(t)}}\,dt. \label{;il}
 \end{align}
 We change variables with $t=H(x),\,x\ge 0$, to rewrite \eqref{;il} as 
 \begin{align}
 &\frac{1}{2}\int _{0}^{\infty }\loc _{x}H'(x)\,dx
 -\int _{0}^{\infty }\frac{1-H(x)}{\loc _{x}}H'(x)\,dx 
 \label{;prf1}\\
 &=\frac{1}{2}\int _{0}^{\infty }\left( \loc _{x}\right) ^{2}dx
 -\int _{0}^{\infty }dx\int _{0}^{1}dt\,\ind _{\{ \bes _{t}>x\} } \notag \\
 &=\frac{1}{2}\int _{0}^{\infty }\left( \loc _{x}\right) ^{2}dx
 -\int _{0}^{1}\bes _{t}\,dt, \notag 
 \end{align}
 where the second line follows from the definition \eqref{;defH} of 
 $H$ and the third from Fubini's theorem. Combining this expression 
 with \eqref{;gauss1} yields 
 \begin{align*}
 \frac{1}{2}\int _{0}^{1}\br _{t}\,dt
 \stackrel{(d)}{=}\frac{1}{2}\int _{0}^{\infty }
 \left( \loc _{x}\right) ^{2}dx-\int _{0}^{1}\bes _{t}\,dt
 \end{align*}
 and concludes the proof. 
\end{proof}

We give a remark on the proof. In what follows we denote 
\begin{align*}
 M(r)=\max _{0\le t\le 1}\bes _{t}. 
\end{align*}
\begin{rem}
\thetag{1}\ We see from the above proof that the random variables 
\begin{align*}
 \int _{0}^{1}\bes _{t}\,dt, \qquad 
 \frac{1}{2}\int _{0}^{\infty }
 \left( \loc _{x}\right) ^{2}dx, \qquad 
 \frac{1}{2}\int _{0}^{1}\frac{1-t}{\bes _{t}}\,dt
\end{align*}
have the same law; they are also identical in law with 
\begin{align*}
 \frac{1}{2}\int _{0}^{t}\frac{t}{\bes _{t}}\,dt
\end{align*}
by the time-reversal: 
$
\{ \bes _{1-t}\} _{0\le t\le 1}
\stackrel{(d)}{=}\{ \bes _{t}\} _{0\le t\le 1}
$. The Laplace transform of the law of 
$\int _{0}^{1}\bes _{t}\,dt$ is given in \cite[Lemma~4.2]{gro} and 
\cite[Proposition~\thetag{5.5}]{by} in terms of a series expansion. \\
\thetag{2}\ We see from \eqref{;defH} that 
\begin{align*}
 \int _{0}^{\infty }\loc _{y}\,dy=\int _{0}^{M(r)}\loc _{y}\,dy=1. 
\end{align*}
Therefore, to be more specific, the second integral in 
\eqref{;prf1} should be written as 
\begin{align*}
 \int _{0}^{M(r)}\frac{1-H(x)}{\loc _{x}}H'(x)\,dx. 
\end{align*}
\end{rem}

Using the same reasoning as the above proof, we may obtain the 
following extension of \tref{;mt}: 
\begin{prop}\label{;pgauss}
 For every positive integer $n$, the random variable 
 \begin{align*}
  2\int _{[0,1]^{n}}\min \left\{ 
  \bes _{t_{1}},\ldots ,\bes _{t_{n}}
  \right\} dt_{1}\cdots dt_{n}
  -\frac{n+1}{2}\int _{0}^{\infty }
  \left( 1-H(x)\right) ^{n-1}\left( \loc _{x}\right) ^{2}dx
 \end{align*}
 has the Gaussian distribution with mean zero and variance 
 $1/(2n+1)$. 
\end{prop}

\begin{proof}
 For each fixed $n$, we multiply both sides of \eqref{;brm} by 
 $(1-t)^{n-1}$ and integrate them over $[0,1]$. Then using 
 Fubini's theorem, we obtain 
 \begin{align}\label{;gauss2}
  \int _{0}^{1}(1-t)^{n-1}\br _{t}\,dt
  =\frac{n+1}{n}\int _{0}^{1}(1-t)^{n-1}\bes _{t}\,dt
  -\frac{1}{n}\int _{0}^{1}\frac{(1-t)^{n}}{\bes _{t}}\,dt. 
 \end{align}
 Since the left-hand side may be expressed as 
 $(1/n)\int _{0}^{1}(1-t)^{n}\,d\br _{t}$, we see that it is a 
 centered Gaussian random variable with variance 
 \begin{align*}
  \frac{1}{n^2}\int _{0}^{1}(1-t)^{2n}\,dt=\frac{1}{n^{2}(2n+1)}. 
 \end{align*}
 On the other hand, by Jeulin's identity \eqref{;ji}, the right-hand side 
 of \eqref{;gauss2} is identical in law with 
 \begin{align*}
  &\frac{n+1}{2n}\int _{0}^{\infty }(1-t) ^{n-1}
  \loc _{H^{-1}(t)}\,dt
  -\frac{2}{n}\int _{0}^{1}\frac{(1-t)^{n}}{\loc _{H^{-1}(t)}}\,dt\\
  &=\frac{n+1}{2n}\int _{0}^{\infty }
  \left( 
  1-H(x)
  \right) ^{n-1}\left( \loc _{x}\right) ^{2}dx
  -\frac{2}{n}\int _{0}^{M(r)}\left( 1-H(x)\right) ^{n}dx. 
 \end{align*}
 By \eqref{;defH}, we may rewrite the integral in the last term as 
 \begin{align*}
  \int _{0}^{M(r)}dx\left( 
  \int _{0}^{1}dt\,\ind _{\{ \bes _{t}>x\} }
  \right) ^{n}
  &=\int _{0}^{M(r)}dx\int _{[0,1]^{n}}
  dt_{1}\cdots dt_{n}\,\prod_{i=1}^{n}\ind _{\{ r_{t_{i}>x}\} }\\
  &=\int _{[0,1]^{n}}\min \left\{ 
  \bes _{t_{1}},\ldots ,\bes _{t_{n}}
  \right\} dt_{1}\cdots dt_{n}, 
 \end{align*}
 where we used Fubini's theorem for the second equality. Combining these 
 leads to the conclusion. 
\end{proof}

We end this note with a comment on a relevant fact deduced from the 
proof of \pref{;pgauss}. 
\begin{rem}
 It is well known (see, e.g., \cite[Equation~\thetag{5d}]{by}) that 
 \begin{align*}
  M(r)\stackrel{(d)}{=}\frac{1}{2}\int _{0}^{1}\frac{dt}{\bes _{t}}; 
 \end{align*}
 in fact, Jeulin's identity \eqref{;ji} entails that 
 \begin{align*}
  \frac{1}{2}\int _{0}^{1}\frac{dt}{\bes _{t}}
  \stackrel{(d)}{=}\int _{0}^{M(r)}\frac{1}{\loc _{x}}\times \loc _{x}\,dx
  =M(r). 
 \end{align*}
 Combining this fact with a part of the proof of \pref{;pgauss}, 
 one sees that the sequence of random variables 
 \begin{align*}
  M(r),\quad \int _{[0,1]}\bes _{t}\,dt,\quad 
  \int _{[0,1]^{2}}\min \left\{ \bes _{t_{1}},\bes _{t_{2}}\right\} 
  dt_{1}dt_{2},\ldots 
 \end{align*}
 is identical in law with 
 \begin{align*}
  \frac{1}{2}\int _{0}^{1}\frac{(1-t)^{n}}{\bes _{t}}\,dt, 
  \quad n=0,1,2,\ldots , 
 \end{align*}
 as well as with 
 \begin{align*}
  \frac{1}{2}\int _{0}^{1}\frac{t^{n}}{\bes _{t}}\,dt, 
  \quad n=0,1,2,\ldots 
 \end{align*}
 by the time-reversal. 
\end{rem}


\end{document}